\documentclass[11pt, reqno]{amsart}
\usepackage{amsmath, amsthm, amscd, amsfonts, amssymb, graphicx, color}
\usepackage[bookmarksnumbered, colorlinks, plainpages]{hyperref}

\makeatletter \oddsidemargin.9375in \evensidemargin \oddsidemargin
\def\authorsaddresses#1{\dedicatory{#1}}
\newtheorem{theorem}{Theorem}[section]

\newtheorem{corollary}[theorem]{Corollary}
\theoremstyle{definition}
\newtheorem{definition}[theorem]{Definition}
\newtheorem{example}[theorem]{Example}

\theoremstyle{remark}

\theoremstyle{approach}

\numberwithin{equation}{section}

\begin{document}
\setcounter{page}{1}


\title[frame related operators for woven frames]{frame related operators for woven frames}

\author[A. Rahimi, Z. Samadzadeh, B. Daraby]{A. Rahimi$ ^{*} $, Z. Samadzadeh$ ^{**} $, B. Daraby$ ^{\dagger} $}

\authorsaddresses{Department ~of Mathematics, University ~of Maragheh, Maragheh, Iran.}
\keywords{Analysis operator, Frame operator, Synthesis operator, Woven frames, Sum of wovens.\\
$ ^{*} $rahimi@maragheh.ac.ir;~$ ^{**} $z.samadzadeh@yahoo.com;~
$ ^{\dagger} $bdaraby@maragheh.ac.ir.}
\begin{abstract}
A new notion in frame theory has been introduced recently under the name woven-weaving frames by Bemrose et. al. In the studying of frames, some operators like analysis, synthesis, Gram and frame operator play the central role. In this paper, for the first time, we introduce and define these operators for woven-weaving frames and review some properties of them. In continuation, we investigate the effect of different types of operators on the woven frames and their bounds. Also, we provide some conditions that shows  sum of woven frames are also woven frames. Finally, we apply these properties with an example.
\end{abstract}

\maketitle

\section{Introduction}
The theory of frames plays an important role in signal processing because of their resilience to quantization \cite{goyal1}, resilience to additive noise, as well as their numerical stability of reconstruction and greater freedom to capture signal characteristics. Also frames have been used in sampling theory to oversampled perfect reconstruction filter banks, system modeling, neural networks and quantum measurements \cite{eldar}. New applications in image processing, robust transmission over the internet and wireless \cite{goyal2}, coding and communication \cite{strohmer} were given.

Discrete frames in Hilbert spaces has been introduced by Duffin and Schaeffer \cite{duffin} and popularized by Daubechies, Grossmann and Meyer \cite{meyer}. A discrete frame is a countable family of elements in a separable Hilbert space which allows stable and not necessarily unique decompositions of arbitrary elements in an expansion of frame elements.
\par
The last two decades have seen tremendous activity in the development of frame theory and many generalizations of frames have come into existence which include bounded quasi-projectors \cite{14}, fusion frames \cite{8}, pseudo-frames \cite{21}, oblique frames \cite{10}, g-frames \cite{30}, continuous frames \cite{27}, K-frames \cite{15}, fractional calculus \cite{16,17}, Hilbert-–Schmidt frames \cite{29} and etc.


In one of the direction of applications of frames in signal processing, a new concept of woven-weaving frames
in a separable Hilbert space introduced by Bemrose et. al. \cite{wov1, wov2}. From the perspective of its introducers, woven frames has potential applications in wireless sensor networks that require distributed data and signal processing. In the field of mathematics of weaving frames, Arefijamaal \cite{jamal1}, Vashisht, Deepshikha and Garg \cite{vd1, vd3, vd4, vd2, vd5} have done more meritorious studies.
By the concepts of weaving, we introduce related operators for weaving and woven frames, and investigate properties of this operators. Also, we study some features of woven frames in Hilbert spaces.

This paper is organized as follows: Section 2 contains the basic definitions of frames and woven frames. Also, we bring some properties of this type of frames that are useful for our study. In section 3, we introduce analysis, synthesis and frame operators of woven and weaving frames and study some properties of these operators. Also some behaviour of woven frames and their sum in Hilbert spaces will be studied in Section 3.

Throughout this paper, $ \mathcal{H} $ is a separable Hilbert space and $ \mathbb{I} $ is the indexing set such that can be finite or infinite countable set. Also, $ [m] $ is the natural numbers set $ \left\lbrace 1, 2, ..., m\right\rbrace $.

\section{Frames and Woven Frames}
As a preliminarily of frames, at the first, we mention some prerequisites of discrete frames and woven-weaving frames in this section.
\subsection{Discrete frame}
We recall definition and some properties of frames that we need in throughout of the paper. For a comprehensive treatment of frame theory, we refer to the excellent textbook by Christensen \cite{ole}.

\begin{definition}
A countable sequence of elements $ \left\lbrace f_{i}\right\rbrace _{i\in\mathbb{I} } $ in $ \mathcal{H}$ is a frame for $ \mathcal{H} $, if there exist constants
$ 0<A, B<\infty $ such that:
\begin{eqnarray}
\label{1}
A\Vert f\Vert^{2} \leq\sum_{i\in\mathbb{I}}\vert\left\langle f, f_{i}\right\rangle \vert^{2} \leq B\Vert f\Vert^{2}, \quad \forall f\in\mathcal{H}.
\end{eqnarray}
The numbers $ A $ and $ B $ are celled frame bounds. The frame
$ \left\lbrace f_{i}\right\rbrace _{i\in\mathbb{I} } $ is called tight frame, if
$ A=B $ and is called Parseval frame if
$ A=B=1 $.
Also the sequence $ \left\lbrace f_{i}\right\rbrace _{i\in\mathbb{I} } $
is called Bessel sequence, if the upper inequality in (\ref{1}) holds.
\end{definition}
For a frame in $ \mathcal{H} $, we define the mapping:
$$ U: \mathcal{H}\longrightarrow\ell^{2}\left(\mathbb{I} \right),\quad U(f)=\left\lbrace \left\langle f, f_{i}\right\rangle \right\rbrace _{i\in\mathbb{I}}. $$
The operator $ U $ is usually called the analysis operator. The adjoint operator of $ U $ is given by:
$$ T:\ell^{2}(\mathbb{I})\longrightarrow\mathcal{H},\quad T\left\lbrace c_{i}\right\rbrace=\sum_{i\in\mathbb{I}}c_{i}f_{i}, $$
and is called the synthesis operator.
By composing $ U $ and $ T $, we obtain the frame operator:
$$ S:\mathcal{H}\longrightarrow\mathcal{H},\quad S(f)= TU(f)=\sum_{i\in\mathbb{I}}\left\langle f, f_{i}\right\rangle f_{i}.$$
The operator $ S $ is positive, self-adjoint and invertible and every $ f\in\mathcal{H} $ can be represented as:
$$ f=\sum_{i\in\mathbb{I}}\left\langle f, S^{-1}f_{i}\right\rangle f_{i}=\sum_{i\in\mathbb{I}}\left\langle f, f_{i}\right\rangle S^{-1}f_{i}. $$
The family $ \left\lbrace S^{-1}f_{i}\right\rbrace _{i\in\mathbb{I}} $ is said the canonical dual frame of
$ \left\lbrace f_{i}\right\rbrace _{i\in\mathbb{I}} $.
\subsection{Woven frames}
Woven frames in Hilbert spaces were introduced by Bemrose et al. \cite{wov1, wov3, wov2} in 2015. After that, Vashisht,  Deepshikha, Arefijamaal and etc. have done more studies over the past few years \cite{jamal1, vd3, vd4, vd2}. In the following, we briefly mention definition of woven frames by presenting an example.
 \begin{definition}
 Let $ F=\left\lbrace f_{ij}\right\rbrace _{i\in\mathbb{I} } $ for $ j\in\left[ m\right]  $ is a family of frames for the separable Hilbert space
 $ \mathcal{H} $. If there exist universal constants $ C$ and $D $, such that for every partition
$ \left\lbrace \sigma_{j}\right\rbrace _{j\in\left[ m\right] } $ of $ \mathbb{I} $ and for every $ j\in[m] $, the family
$ F_{j}=\left\lbrace f_{ij}\right\rbrace _{i\in\sigma_{j}} $ is a frame for
$ \mathcal{H} $ with bounds
$ C$ and $D $, then $ F $ is said a woven frames. For every $ j\in\left[ m\right]  $, the frames
$ F_{j}=\left\lbrace f_{ij}\right\rbrace _{i\in\sigma_{j}} $ are said a weaving frame.

The constants $ C $ and $ D $ are called the lower and upper woven frame bounds. If
$ C=D $, then
$ F=\left\lbrace f_{ij}\right\rbrace _{i\in\mathbb{I}, j\in\left[ m\right] } $ is said a tight woven frame and if for every
$ j\in\left[ m\right]  $, the family $ F_{j}=\left\lbrace f_{ij}\right\rbrace _{i\in\sigma_{j}} $ is a Bessel sequence, then
the family $ F=\left\lbrace f_{ij}\right\rbrace _{i\in\mathbb{I},j\in\left[ m\right] } $ is said a Bessel woven.
 \end{definition}
In the continuation, we dissect woven and weaving frames for $ m=2 $.
\begin{example}
Let $ \left\lbrace e_{i}\right\rbrace _{i=1}^{2} $ be an orthonormal basis for the two dimensional vector space
$ V=\overline{\rm span}\left\lbrace e_{i}\right\rbrace _{i=1}^{2} $ with inner product and suppose that $ F $ and
$ G $ are the sets:
$$ F=\left\lbrace 2e_{1},2e_{2}-e_{1},3e_{2}\right\rbrace \quad ,\quad G=\left\lbrace 2e_{1},2e_{1}+e_{2},2e_{2}\right\rbrace . $$
Since both of $ F $ and $ G $ span the space $ V $, then those are frames. For obtaining their bounds, we have:
\begin{eqnarray*}
\sum_{i=1}^{3}\left\vert \left\langle f, f_{i}\right\rangle \right\vert^{2}&=&
\left\vert \left\langle f, 2e_{1}\right\rangle \right\vert^{2}+\left\vert \left\langle f, 2e_{2}-e_{1}\right\rangle \right\vert^{2}+
\left\vert \left\langle f, 3e_{2}\right\rangle \right\vert^{2},
\end{eqnarray*}
then
$$ 4 \Vert f\Vert^{2}\leq\sum_{i=1}^{3}\left\vert \left\langle f, f_{i}\right\rangle \right\vert^{2}\leq 17\Vert f\Vert^{2}, $$
thus $ F=\left\lbrace f_{i}\right\rbrace_{i=1}^{3}  $ is a frame for $ V $ with lower bound $ 4 $ and upper bound $ 17 $. Similarly,
$ G=\left\lbrace g_{i}\right\rbrace_{i=1}^{3} $ is a frame for $ V $ with frame bounds $ 4 $ and $ 9 $. The sets $ F $ and $ G $ are woven frames for $ V $. For example, if $ \sigma_{1}=\left\lbrace 1,2\right\rbrace  $, then for every $ f\in V $, we have
$$ 4\left\Vert f\right\Vert^{2}\leq \sum_{i\in\sigma_{1}}\left\vert \left\langle f, f_{i}\right\rangle \right\vert^{2}+\sum_{i\in\sigma_{1}^{c}}\left\vert \left\langle f, g_{i}\right\rangle \right\vert^{2}\leq 12\left\Vert f\right\Vert^{2} . $$
So $ \left\lbrace f_{i}\right\rbrace _{i\in\sigma_{1}}\bigcup \left\lbrace g_{i}\right\rbrace _{i\in\sigma_{1}^{c}} $
is a weaving frame with bounds $ C_{1}=4 $ and $ D_{1}=12 $.\\
 Now, if we take:
$$ C=\min \left\lbrace  C_{i};\quad 1\leq i\leq 8\right\rbrace \quad, \quad D=\max \left\lbrace  D_{i};\quad 1\leq i\leq 8\right\rbrace , $$
then $ F $ and $ G $ form a woven frames for $ V $ with universal bounds $ C $ and $ D $.
\section{Related Operators for Woven and Weaving Frames}
In this section, for first time, we introduce analysis, synthesis and frame operators of weaving and woven frames. As a prerequisite for this operators, we define the following space.

\end{example}

For each family of subspaces $ \left\lbrace\left( \ell^{2}(\mathbb{I} )\right)_{j}  \right\rbrace _{j\in[m]} $ of
$ \ell^{2}(\mathbb{I} ) $, we have
$$ \left( \ell^{2}(\mathbb{I} )\right)_{j}=\left\lbrace \left\lbrace c_{ij}\right\rbrace _{i\in\sigma_{j}}\vert ~c_{ij}\in \mathbb{C} ~,~\sigma_{j}\subset \mathbb{I},~ \sum_{i\in\sigma_{j}}\left\vert c_{i}\right\vert^{2}<\infty\right\rbrace \quad,\quad \forall j\in [m]. $$
We define the space:
$$ \left( \sum_{j\in[m]}\bigoplus \left( \ell^{2}(\mathbb{I} )\right)_{j}\right) _{\ell_{2}}=\left\lbrace \left\lbrace c_{ij}\right\rbrace _{i\in\mathbb{I},j\in[m]}\vert  \left\lbrace c_{ij}\right\rbrace _{i\in\mathbb{I}}\in \left( \ell^{2}(\mathbb{I} )\right)_{j}, \forall j\in [m]\right\rbrace , $$
with the inner product
$$ \left\langle \left\lbrace c_{ij}\right\rbrace _{i\in\mathbb{I}, j\in[m]}, \left\lbrace c^{\prime }_{ij}\right\rbrace _{i\in\mathbb{I},j\in[m]}\right\rangle =\sum_{i\in\mathbb{I},j\in[m]} \left\vert c_{ij}\overline{c_{ij}^{\prime}}\right\vert , $$
it is easy to show that this space is a Hilbert space.
\begin{theorem}
\label{T8}
The family $ \left\lbrace f_{ij}\right\rbrace _{i\in\mathbb{I}, j\in\left[ m\right]} $ is a Bessel woven if and only if the operator
$$  T_{F}:\left( \sum_{j\in[m]}\bigoplus \left( \ell^{2}(\mathbb{I} )\right)_{j}\right) _{\ell_{2}}\longrightarrow\mathcal{H}\quad ,\quad T_{F}\left\lbrace c_{ij}\right\rbrace _{i\in\mathbb{I},j\in[m]}=\sum_{i\in\mathbb{I},j\in[m]}c_{ij}f_{ij} $$
 is well defined, linear and bounded.
\end{theorem}
\begin{proof}
Let $ \left\lbrace f_{ij}\right\rbrace _{i\in\mathbb{I}, j\in\left[ m\right]} $ be a Bessel woven. Then for every fix $ j\in\left[ m\right]  $ and
$ \sigma_{j}\subset\mathbb{I}$, the family
$ \left\lbrace f_{ij}\right\rbrace _{i\in\sigma_{j}}$ is a Bessel sequence with Bessel bound $ D_{j} $. If we show the synthesis operator of $ \left\lbrace f_{ij}\right\rbrace _{i\in\sigma_{j}}$ with $ T_{\sigma_{j}} $, thus for every $ \left\lbrace c_{ij}\right\rbrace_{i\in \mathbb{I}, j\in [m]} \in \left( \sum_{j\in[m]}\bigoplus \left( \ell^{2}(\mathbb{I} )\right)_{j}\right) _{\ell_{2}} $, we have:
\begin{eqnarray*}
\Vert T_{F}\left\lbrace c_{ij}\right\rbrace\Vert^{2}&=&\left\Vert \sum_{i\in\mathbb{I},j\in[m]}c_{ij}f_{ij} \right\Vert^{2}\\
&=& \left\Vert \sum_{i\in\mathbb{I}}c_{i1}f_{i1}+\sum_{i\in\mathbb{I}}c_{i2}f_{i2}+ ... + \sum_{i\in\mathbb{I}}c_{im}f_{im} \right\Vert^{2}\\
&\leq & 2\left\Vert T_{\sigma_{1}}\left\lbrace c_{i1}\right\rbrace _{i\in \sigma_{1}}\right\Vert^{2}+2\left\Vert T_{\sigma_{2}}\left\lbrace c_{i2}\right\rbrace_{i\in \sigma_{2}} \right\Vert^{2}+ ... + 2\left\Vert T_{\sigma_{m}}\left\lbrace c_{im}\right\rbrace_{i\in \sigma_{m}} \right\Vert^{2}\\
&\leq & 2D_{1}\left\Vert \left\lbrace c_{i1}\right\rbrace_{i\in \sigma_{1}} \right\Vert^{2}+2D_{2}\left\Vert \left\lbrace c_{i2}\right\rbrace_{i\in \sigma_{2}} \right\Vert^{2}+ ... +2D_{m}\left\Vert \left\lbrace c_{im}\right\rbrace_{i\in \sigma_{m}} \right\Vert^{2}\\
&\leq & 2mD\left\Vert \left\lbrace c_{ik}\right\rbrace _{i\in \sigma_{k}}\right\Vert^{2}\\
&\leq & 2mD\left\Vert \left\lbrace c_{ij}\right\rbrace _{i\in \mathbb{I}, j\in [m]}\right\Vert^{2}\
\end{eqnarray*}
where
 $$ D=\max \left\lbrace D_{j};\quad 1\leq j\leq m\right\rbrace ~, ~
\left\Vert \left\lbrace c_{ik}\right\rbrace_{i\in \sigma_{k}} \right\Vert=\max \left\lbrace \left\Vert \left\lbrace c_{ij}\right\rbrace_{i\in \sigma_{j}} \right\Vert , 1\leq j\leq m\right\rbrace . $$
This calculation shows that $ T_{F} $ is bounded and well defined. It is clear that $ T_{F} $ is linear.

Conversely, suppose $ T_{F} $ is well defined, linear and bounded with bound $ D $. For every
$ f\in\mathcal{H} $, we have:
$$ \left\langle T_{F}\left\lbrace c_{ij}\right\rbrace_{i\in \mathbb{I}, j\in [m] } , f \right\rangle =\left\langle \sum_{i\in \mathbb{I}, j\in [m] }c_{ij}f_{ij}, f\right\rangle =\left\langle \left\lbrace c_{ij}\right\rbrace_{i\in \mathbb{I}, j\in [m] } , \left\lbrace \left\langle f, f_{ij}\right\rangle \right\rbrace_{i\in \mathbb{I}, j\in [m] } \right\rangle , $$
therefore
$$ T^{*}_{F}f=\left\lbrace \left\langle f, f_{ij}\right\rangle \right\rbrace_{i\in \mathbb{I}, j\in [m] }, $$
and
\begin{eqnarray*}
\sum_{i\in\mathbb{I}, j\in\left[ m\right] }\left\vert \left\langle f, f_{ij}\right\rangle \right\vert ^{2}&=&
\left\Vert T^{*}_{F}f\right\Vert ^{2}\\
&\leq & \left\Vert T^{*}_{F}\right\Vert^{2} \left\Vert f\right\Vert ^{2}\\
&=& \left\Vert T_{F}\right\Vert^{2} \left\Vert f\right\Vert ^{2}\\
&\leq & D\left\Vert f\right\Vert ^{2}.
\end{eqnarray*}
So the family $ \left\lbrace f_{ij}\right\rbrace _{i\in\mathbb{I}, j\in\left[m \right] } $ is a Bessel woven.
\end{proof}
If we only need to know that $ \left\lbrace f_{ij}\right\rbrace _{i\in\mathbb{I}, j\in\left[ m\right]} $ is a Bessel woven and the value of the Bessel bound is irrelevant, we need to check that the operator $ T_{F} $ is well defined or not.
\begin{corollary}
If $ \left\lbrace f_{ij}\right\rbrace _{i\in\mathbb{I}, j\in\left[ m\right]} $ is a family in $ \mathcal{H} $ and the series
$ \sum_{i\in\mathbb{I}, j\in\left[ m\right] }c_{ij}f_{ij} $ is convergent for all
$ \left\lbrace c_{ij}\right\rbrace _{i\in\mathbb{I}, j\in\left[ m\right]}\in\left( \sum_{j\in[m]}\bigoplus \left( \ell^{2}(\mathbb{I} )\right)_{j}\right) _{\ell_{2}} $. Then the sequence
$ \left\lbrace f_{ij}\right\rbrace _{i\in\mathbb{I}, j\in\left[ m\right]} $ is a Bessel woven.
\end{corollary}

The Bessel woven condition:
$$ \sum_{i\in\mathbb{I}, j\in\left[ m\right] }\left\vert \left\langle f, f_{ij}\right\rangle  \right\vert^{2}\leq D\Vert f\Vert^{2},\quad \forall f\in\mathcal{H}, $$
remains the same, regardless of how the elements $ \left\lbrace f_{ij}\right\rbrace _{i\in\mathbb{I}, j\in\left[ m\right]} $ are numbered.
\begin{corollary}
\label{C1}
Suppose that the family $ \left\lbrace f_{ij}\right\rbrace _{i\in\mathbb{I}, j\in\left[ m\right]} $ is a Bessel woven for $ \mathcal{H} $. Then, the series $ \sum_{i\in\mathbb{I}, j\in\left[ m\right] }c_{ij}f_{ij} $ converges unconditionally for all
$ \left\lbrace c_{ij}\right\rbrace _{i\in\mathbb{I}, j\in\left[ m\right]}\in \left( \sum_{j\in[m]}\bigoplus \left( \ell^{2}(\mathbb{I} )\right)_{j}\right) _{\ell_{2}} $.
\end{corollary}
\begin{proof}
Since $ \left\lbrace f_{ij}\right\rbrace _{i\in\mathbb{I}, j\in\left[ m\right]} $ is a Bessel woven, then by Theorem \ref{T8}, \cite{ole}, we have:
$$ \sum_{i\in\mathbb{I},j\in\left[ m\right] }\left\vert\left\langle f, f_{ij}\right\rangle\right\vert^{2}\leq D\Vert f\Vert^{2},\quad \forall f\in\mathcal{H}, $$
which is equivalent to
$$\left\Vert T_{F}\lbrace c_{ij}\rbrace\right\Vert^{2}\leq D\sum_{i\in\mathbb{I},j\in\left[ m\right]}\vert c_{ij}\vert^{2},\quad \forall\left\lbrace c_{ij}\right\rbrace _{i\in\mathbb{I}, j\in\left[ m\right]}\in \left( \sum_{j\in[m]}\bigoplus \left( \ell^{2}(\mathbb{I} )\right)_{j}\right) _{\ell_{2}} . $$
Therefore the series $ \sum_{i\in\mathbb{I},j\in\left[ m\right] }c_{ij}f_{ij} $ is convergent, so there exists $ f\in\mathcal{H} $ such that
$$ \left\Vert \sum_{i\in\mathbb{I},j\in\left[ m\right] }c_{ij}f_{ij} -f \right\Vert \leq\frac{\varepsilon}{2} .$$
Now, we show that the series
$ \sum_{i\in\mathbb{I},j\in\left[ m\right] }c_{ij}f_{ij} $ converges unconditionally. By Cauchy-Schwarz inequality, we have:
\begin{eqnarray*}
\left\Vert T_{F}\lbrace c_{ij}\rbrace_{i\in\mathbb{I}, j\in [m]}\right\Vert^{2}&=&
\sup_{\Vert g\Vert =1}\left\vert \left\langle \sum_{i\in\mathbb{I},j\in\left[ m\right]}c_{ij}f_{ij},g\right\rangle \right\vert\\
&\leq & \sup_{\Vert g\Vert =1}\sum_{i\in\mathbb{I},j\in\left[ m\right]}\left\vert c_{ij} \left\langle f_{ij},g\right\rangle \right\vert\\
&\leq &\left( \sum_{i\in\mathbb{I},j\in\left[ m\right]}\vert c_{ij}\vert^{2}\right) ^{\frac{1}{2}}
\sup_{\Vert g\Vert =1}\left( \sum_{i\in\mathbb{I},j\in\left[ m\right]}\left\vert\left\langle g, f_{ij}\right\rangle \right\vert^{2}\right) ^{\frac{1}{2}}.
\end{eqnarray*}
Since $ \left\lbrace f_{ij}\right\rbrace_{i\in\mathbb{I},j\in\left[ m\right] }^{\infty} $ is Bessel sequence, then
$ \sum_{i\in\mathbb{I},j\in\left[ m\right] }^{\infty}\left\vert\left\langle g, f_{ij}\right\rangle \right\vert^{2} $ is finite. Then for every $ \varepsilon >0 $, there exists
$ M\in\mathbb{I} $
$$ \sum_{i=M+1,j\in\left[ m\right] }^{\infty}\left\vert\left\langle g, f_{ij}\right\rangle \right\vert^{2}< \frac{\varepsilon}{2} .$$
Then
$$ \left\Vert\sum_{i=M+1,j\in\left[ m\right] }^{\infty}c_{ij}f_{ij}\right\Vert < \frac{\varepsilon}{2}. $$
So we have
$$  \left\Vert\sum_{i=1,j\in\left[ m\right] }^{M}c_{ij}f_{ij}-f\right\Vert -
 \left\Vert\sum_{i=M+1,j\in\left[ m\right] }^{\infty}c_{ij}f_{ij}\right\Vert \leq \left\Vert\sum_{i=1,j\in\left[ m\right] }^{\infty}c_{ij}f_{ij}-f \right\Vert\\
\leq \frac{\varepsilon}{2}. $$
Therefor
\begin{eqnarray*}
\left\Vert\sum_{k=1,j\in\left[ m\right] }^{M}c_{i_{k}j}f_{i_{k}j}-f\right\Vert < \varepsilon ,\quad \forall\varepsilon>0 .
\end{eqnarray*}
If we suppose $ F=\left\lbrace 1, 2, ... , M\right\rbrace $, we have
$$ \left\Vert\sum_{i\in F,j\in\left[ m\right] }c_{ij}f_{ij}-f\right\Vert < \varepsilon\quad ,\quad \varepsilon >0 . $$
So by Lemma 2.1.1, \cite{ole}, the series $ \sum_{i\in\mathbb{I}, j\in\left[ m\right] }c_{ij}f_{ij}$ converges unconditionally.
\end{proof}
Like frames and its extensions, we can characterize a woven frame in term of its woven frame operator.
\begin{definition}
Let $ F=\left\lbrace f_{ij}\right\rbrace _{i\in\mathbb{I},j\in\left[ m\right] } $ be a woven Bessel. Then for every partition
$ \left\lbrace \sigma_{j}\right\rbrace _{j\in\left[ m\right] } $, the family
$ F_{j}=\left\lbrace f_{ij}\right\rbrace _{i\in\sigma_{j}} $ for $ j\in\left[ m\right]  $ is a Bessel sequence. Therefore,
we define the analysis operator of $ F_{j} $ by
$$ U_{\sigma_{j}}: \mathcal{H}\longrightarrow \left( \ell^{2}(\mathbb{I} )\right)_{j},\quad U_{\sigma_{j}}(f)=\left\lbrace \left\langle f, f_{ij}\right\rangle \right\rbrace _{i\in\sigma_{j}},\quad \forall j\in[m], f\in\mathcal{H},$$
then $ Ran\left( U_{\sigma_{j}}\right) \subseteq \left( \ell^{2}(\mathbb{I} )\right)_{j}\subseteq \ell^{2}(\mathbb{I}) $.
The adjoint of $  U_{\sigma_{j}} $ is called the synthesis operator and in this paper, we denote
$ T_{\sigma_{j}}=U_{\sigma_{j}}^{*} $. By elementary calculation, for every $ j\in [m] $, we have:
$$ T_{\sigma_{j}}: \left( \ell^{2}(\mathbb{I})\right) _{j}\longrightarrow\mathcal{H},\quad T_{\sigma_{j}}\left\lbrace c_{ij}\right\rbrace_{i}= \sum_{i\in\sigma_{j}} c_{ij}f_{ij},~ \forall \left\lbrace c_{ij}\right\rbrace_{i} \in \left(\ell^{2}(\mathbb{I})\right)_{j}. $$
The frame operator of a weaving Bessel is obtained by combination of analysis and synthesis operators. For every
 $ f\in\mathcal{H}$ and $ j\in [m] $:
$$ S_{\sigma_{j}}f= T_{\sigma_{j}}U_{\sigma_{j}}f=T_{\sigma_{j}}\left\lbrace \left\langle f, f_{ij}\right\rangle \right\rbrace _{i\in\sigma_{j}}=\sum_{i\in\sigma_{j}}\left\langle f, f_{ij}\right\rangle f_{ij},$$
The operator $ S_{\sigma_{j}} $ is bounded, self-adjoint and invertible. We call the family
$ \left\lbrace S^{-1}_{\sigma_{j}}f_{ij}\right\rbrace _{i\in\sigma_{j}} $ standard dual weaving frame of $ F_{j} $.
Now, we define the analysis and synthesis operators for the Bessel woven $ F=\left\lbrace f_{ij}\right\rbrace _{i\in\mathbb{I}, j\in\left[ m\right]} $:
$$ U_{F}:\mathcal{H}\longrightarrow \left( \sum_{j\in[m]}\bigoplus \left( \ell^{2}(\mathbb{I} )\right)_{j}\right) _{\ell_{2}},\quad U_{F}(f)=\left\lbrace \left\langle f, f_{ij}\right\rangle \right\rbrace _{i\in\mathbb{I},j\in[m]}, $$
and
$$ T_{F}:\left( \sum_{j\in[m]}\bigoplus \left( \ell^{2}(\mathbb{I} )\right)_{j}\right) _{\ell_{2}}\longrightarrow \mathcal{H},\quad T_{F}\left\lbrace c_{ij}\right\rbrace _{i\in\mathbb{I},j\in[m]}=\sum_{i\in\mathbb{I},j\in[m]}c_{ij}f_{ij} . $$
The operators $ U_{F} $ and $ T_{F} $ are well defined and bounded analysis and synthesis operators, respectively.
Also, by combination of $ U_{F} $ and $ T_{F} $, the woven frame operator $ S_{F} $, for all $ f\in\mathcal{H} $, is defined by
$$ S_{F}:\mathcal{H}\longrightarrow\mathcal{H},\quad S_{F}f=T_{F}U_{F}f=\sum_{i\in\mathbb{I},j\in[m]}\left\langle f, f_{ij}\right\rangle f_{ij}.$$
The operator $ S_{F} $ is bounded, linear and self-adjoint operator. Also every $ f\in\mathcal{H} $ can be represented as
$$ f=\sum_{i\in\mathbb{I},j\in[m]}\left\langle f, S_{F}^{-1}f_{ij}\right\rangle f_{ij}=\sum_{i\in\mathbb{I},j\in[m]}\left\langle f, f_{ij}\right\rangle S_{F}^{-1}f_{ij}. $$
The family $ \left\lbrace S^{-1}_{F}f_{ij}\right\rbrace_{i\in\mathbb{I},j\in[m]}  $ is said the standard dual woven of
$ F $.
\end{definition}
\subsection{Main results}
In this subsection, we explain some properties of operators of woven frames and Bessel wovens. Also, we bring some conditions that under those, the sum of Bessel wovens shall be woven frames.

In the next theorem, we demonstrate that the woven frames are equivalent to boundedness of woven frame operator.
\begin{theorem}
Let $ \left\lbrace f_{ij}\right\rbrace _{i\in\mathbb{I}, j\in\left[ m\right]} $ be finite family of Bessel sequences in $ \mathcal{H} $. Then the following conditions are equivalent:
\begin{enumerate}
\item[(i)] $ \left\lbrace f_{ij}\right\rbrace _{i\in\mathbb{I}, j\in\left[ m\right]} $ is woven frames with universal woven frame bounds $ C $ and $ D $.
\item[(ii)] for the operator $ S_{F}f=\sum_{i\in\mathbb{I},j\in\left[ m\right]}\left\langle f, f_{ij}\right\rangle f_{ij} $, we have $ CI_{\mathcal{H} }\leq S_{F}\leq DI_{\mathcal{H} } $.
\end{enumerate}
\end{theorem}
\begin{proof}
$ \rm i\Rightarrow \rm ii $: By statement of $ (\rm i) $, we have for every $ f\in\mathbb{I} $
$$ \left\langle S_{F}f,f\right\rangle =\left\langle \sum_{i\in\mathbb{I},j\in\left[ m\right]}\left\langle f,f_{ij}\right\rangle f_{ij}, f\right\rangle =
\sum_{i\in\mathbb{I},j\in\left[ m\right]}\left\vert \left\langle f,f_{ij}\right\rangle\right\vert^{2} , $$
then we obtain
$$ C\left\Vert f\right\Vert^{2}\leq \left\langle S_{F}f,f\right\rangle \leq D\left\Vert f\right\Vert^{2},\quad \forall f\in \mathcal{H}, $$
or equivalently
$$ CI_{\mathcal{H}}\leq S_{F}\leq DI_{\mathcal{H}} . $$
$ \rm ii\Rightarrow\rm i $: Let
$ U_{F} $ be the analysis operator associated with
$ \left\lbrace f_{ij}\right\rbrace _{i\in\sigma_{j}, j\in\left[ m\right]} $. By the fact that
$ \left\Vert S_{F} \right\Vert = \left\Vert T_{F}U_{F} \right\Vert =\left\Vert U_{F} \right\Vert ^{2} $, for all $ f\in\mathcal{H} $, we have
\begin{eqnarray*}
\sum_{i\in\mathbb{I},j\in\left[ m\right]}\left\vert \left\langle f,f_{ij}\right\rangle\right\vert^{2}&=& \left\Vert U_{F}f\right\Vert ^{2}\\
&\leq & \left\Vert U_{F} \right\Vert ^{2} \left\Vert f\right\Vert ^{2}\\
&=& \left\Vert S_{F} \right\Vert \left\Vert f\right\Vert ^{2}\\
&\leq & D\left\Vert f\right\Vert ^{2}.
\end{eqnarray*}
So the upper woven bound is established. For lower weaving bound, for all $ f\in\mathcal{H} $, we have
\begin{eqnarray*}
\sum_{i\in\mathbb{I},j\in\left[ m\right]}\left\vert \left\langle f,f_{ij}\right\rangle\right\vert^{2}&=&
\left\langle S_{F}f,f\right\rangle\\
&=& \left\langle S_{F}^{\frac{1}{2}}f,S_{F}^{\frac{1}{2}}f\right\rangle\\
&=& \left\Vert S_{F}^{\frac{1}{2}}f \right\Vert^{2}\\
&\geq & C\left\Vert f \right\Vert^{2}.
\end{eqnarray*}
\end{proof}
The next result shows that we can constitute tight woven frames from every woven frames by weaving operators.
\begin{theorem}
Let $ F=\left\lbrace f_{ij}\right\rbrace _{i\in\mathbb{I}, j\in\left[ m\right]} $ be woven frame for $ \mathcal{H} $ with universal woven bounds $ C $ and $ D $ and the woven frame operator $ S_{F} $. If we define the positive square root of $ S^{-1}_{F} $ with $ S^{-\frac{1}{2}}_{F} $, then
$ \left\lbrace S^{-\frac{1}{2}}_{F}f_{ij}\right\rbrace _{i\in\mathbb{I}, j\in\left[ m\right]  } $ is tight woven frame and for all $ f\in\mathcal{H} $, we have:
$$ f=\sum_{i\in\mathbb{I}, j\in\left[ m\right]}\left\langle f, S^{-\frac{1}{2}}_{F} f_{ij}\right\rangle S^{-\frac{1}{2}} _{F} f_{ij}. $$
\end{theorem}
\begin{proof}
Since the sequence $ \left\lbrace f_{ij}\right\rbrace _{i\in\mathbb{I}, j\in\left[ m\right]} $ is a woven frame with bounds $ C$ and $D $, we can write
$ CI_{\mathcal{H}}\leq S_{F}\leq DI_{\mathcal{H}} $, thus
$ D^{-1}I_{\mathcal{H}}\leq S^{-1}_{F}\leq C^{-1}I_{\mathcal{H}} $.
Now, by definition of woven frame operator of
$ F=\left\lbrace f_{ij}\right\rbrace _{i\in\sigma_{j}, j\in\left[ m\right]  } $, we have for every $ f\in\mathcal{H} $:
$$ S_{F}f=\sum_{i\in\mathbb{I}, j\in\left[ m\right]}\left\langle f, f_{ij}\right\rangle f_{ij} . $$
By substitution of $ S^{-\frac{1}{2}}_{F}f $, we have:
$$ S^{\frac{1}{2}}_{F}f =\sum_{i\in\mathbb{I}, j\in\left[ m\right]}\left\langle S^{-\frac{1}{2}}_{F}f, f_{ij}\right\rangle f_{ij} , $$
 therefore
\begin{eqnarray*}
f&=&S^{-\frac{1}{2}}_{F}\left( S^{\frac{1}{2}}_{F}f\right)\\&=&
S^{-\frac{1}{2}}_{F}\left( \sum_{i\in\mathbb{I}, j\in\left[ m\right]}\left\langle f, S^{-\frac{1}{2}}_{F}f_{ij}\right\rangle f_{ij}\right)\\
&=& \sum_{i\in\mathbb{I}, j\in\left[ m\right]}\left\langle f, S^{-\frac{1}{2}}_{F}f_{ij}\right\rangle S^{-\frac{1}{2} }_{F}f_{ij},
\end{eqnarray*}
Thus we obtain
\begin{eqnarray*}
\Vert f\Vert^{2}&=& \left\langle \sum_{i\in\mathbb{I}, j\in\left[ m\right]}\left\langle f, S^{-\frac{1}{2}}_{F} f_{ij}\right\rangle S^{-\frac{1}{2}} _{F} f_{ij}, f\right\rangle \\
&=&  \sum_{i\in\mathbb{I}, j\in\left[ m\right]}\left\vert \left\langle f, S^{-\frac{1}{2}}_{F} f_{ij}\right\rangle \right\vert^{2}.
\end{eqnarray*}
So $ \left\lbrace S^{-\frac{1}{2}}_{F} f_{ij}\right\rbrace _{i\in\mathbb{I}, j\in\left[ m\right]} $ is tight woven frames with universal bound 1.
\end{proof}
In the following theorem, we investigate the effect of a bounded and invertible operator on woven frames.
\begin{theorem}
\label{T3}
Let $ F=\left\lbrace f_{ij}\right\rbrace _{i\in\mathbb{I},j\in\left[ m\right]  } $ be a woven frame for $ \mathcal{H} $ with woven frame operator $ S_{F} $ and universal bounds
$ C $ and $ D $ and $ E:\mathcal{H}\longrightarrow\mathcal{H} $ be a bounded operator. Then the operator $ E $ is invertible if and only if $ \left\lbrace Ef_{ij}\right\rbrace _{i\in\mathbb{I},j\in\left[ m\right]} $ is woven frame for $ \mathcal{H} $. In this case, the universal bounds for $ F $ are
$ C\left\Vert E^{-1}\right\Vert^{-2}, D\left\Vert E\right\Vert^{2} $ and the woven frame operator $ ES_{F}E^{*} $.
\end{theorem}
\begin{proof}
Let $ E $ be invertible, since $ \left\lbrace f_{ij}\right\rbrace _{i\in\mathbb{I},j\in\left[ m\right]} $ is a woven frames for $ \mathcal{H} $,  with bounds $ C $ and $ D $, so the boundedness of $ T $ verifies the upper bound:
$$\sum_{i\in\mathbb{I}, j\in\left[ m\right] }\left\vert \left\langle f,Ef_{ij} \right\rangle \right\vert^{2}\leq
D\left\Vert E^{\ast}f\right\Vert^{2}\leq
D\left\Vert E^{\ast}\right\Vert^{2} \left\Vert f\right\Vert^{2}=
D\left\Vert E\right\Vert^{2} \left\Vert f\right\Vert^{2}.$$
For lower bound, we assume that $ g\in\mathcal{H} $. Since $ E $ is surjective, there exists $ f\in\mathcal{H} $ such that $ Ef=g $. Therefore
\begin{eqnarray*}
\left\Vert g\right\Vert^{2}&=&\left\Vert Ef\right\Vert^{2}\\
&=&\left\Vert \left( EE^{-1}\right) ^{*}Ef\right\Vert^{2}\\
&=&\left\Vert \left( E^{-1}\right) ^{*}E^{*}Ef\right\Vert^{2}\\
&\leq & \left\Vert E^{-1}\right\Vert^{2}\left\Vert E^{*}Ef\right\Vert^{2}\\
&\leq & \frac{\left\Vert E^{-1}\right\Vert^{2}}{C}\sum_{i\in\mathbb{I} ,j\in\left[ m\right]}
\left\vert \left\langle E^{*}Ef,f_{ij}\right\rangle \right\vert^{2}\\
&=& \frac{\left\Vert E^{-1}\right\Vert^{2}}{C}\sum_{i\in\mathbb{I} ,j\in\left[ m\right]}
\left\vert \left\langle Ef,Ef_{ij}\right\rangle \right\vert^{2},
\end{eqnarray*}
so we have:
$$  C\left\Vert E^{-1}\right\Vert^{-2}\left\Vert g\right\Vert^{2}\leq\sum_{i\in\mathbb{I} ,j\in\left[m\right]}\left\vert \left\langle g,Ef_{ij}\right\rangle \right\vert^{2}. $$
This calculations provide the desired results.

Conversely, If $ \left\lbrace Ef_{ij}\right\rbrace _{i\in\mathbb{I},j\in\left[ m\right]} $ is a woven frames for $ \mathcal{H} $ then its woven frame operator is invertible
$$ \sum_{i\in\mathbb{I},j\in\left[ m\right]}\left\langle f, Ef_{ij}\right\rangle Ef_{ij}=E\left( \sum_{i\in\mathbb{I},j\in\left[ m\right]}\left\langle f, Ef_{ij}\right\rangle f_{ij}\right)=ES_{F}E^{*}f,  $$
which implies that $ E $ is invertible.
\end{proof}
From the previous theorem, we can obtain the next result.
\begin{theorem}
Let $ F=\left\lbrace f_{ij}\right\rbrace _{i\in\mathbb{I},j\in [m] } $ be woven frame with universal woven bounds $ C $ and $ D $ for $ \mathcal{H} $ and woven frame operator $ S_{F} $. Then for every $ \sigma_{j}\subset\mathbb{I} $,
$  j\in\left[ m\right] $,we have:
\begin{enumerate}
\item[(i)] The sequence $ \left\lbrace S_{\sigma_{j}}f_{ij}\right\rbrace _{i\in\sigma_{j}} $ for every $ j\in\left[ m\right]  $ is a weaving frame i.e. the family
$ \left\lbrace S_{F}f_{ij}\right\rbrace _{i\in\mathbb{I}, j\in\left[ m\right] } $ is woven frames for $\mathcal{H}$, with universal lower and upper woven bounds $ C^3$ and $ D^{3}$, respectively.
\item[(ii)] The sequence $ \left\lbrace S_{\sigma_{j}}^{-1}f_{ij}\right\rbrace _{i\in\sigma_{j} } $ for every $ j\in\left[ m\right] $ is a weaving frame i.e. the family $ \left\lbrace S_{F}^{-1}f_{ij}\right\rbrace _{i\in\mathbb{I}, j\in\left[ m\right] } $ is a  woven frames for $\mathcal{H}$, with universal lower and upper woven bounds $ \frac{C }{D^{2}}$ and $ \frac{D}{C^{2}}$, respectively.
\end{enumerate}
\end{theorem}
\begin{proof}
By Theorem $ \ref{T3} $, if we put
$ E=S_{\sigma_{j}}, S_{F} $ and $ E=S_{\sigma_{j}}^{-1}, S_{F}^{-1} $, then provide the results.
\end{proof}
In the next theorem, we give some conditions that under those, sum of wovens are woven also.
\begin{theorem}
\label{sumT}
Let $ F=\left\lbrace f_{ij}\right\rbrace _{i\in\mathbb{I},j\in[m]} $ and
$ G=\left\lbrace g_{ij}\right\rbrace _{i\in\mathbb{I},j\in[m]} $ be woven Bessel sequences in $ \mathcal{H} $, with analysis operators $ U_{F} $, $ U_{G} $ and woven frame operators $ S_{F} $, $ S_{_{G}}$, respectively. Also, let
$ E_{1},E_{2}:\mathcal{H}\longrightarrow\mathcal{H} $ be bounded operators. Then the family
$ \left\lbrace E_{1}f_{ij}+E_{2}g_{ij}\right\rbrace _{i\in\mathbb{I},j\in\left[ m\right] } $ is a woven frames for $ \mathcal{H} $ if and only if the operator $ U_{F}E_{1}^{*}+U_{G}E_{2}^{*} $ is an invertible operator on $ \mathcal{H} $.
\end{theorem}
\begin{proof}
The family $ \left\lbrace E_{1}f_{ij}+E_{2}g_{ij}\right\rbrace _{i\in\mathbb{I},j\in\left[ m\right] } $ is a woven frames if and only if its analysis operator
$ U $ is invertible on $ \mathcal{H} $, by Theorem 4.1 in \cite{casazza}. Therefor
\begin{eqnarray*}
Uf&=&\left\lbrace \left\langle f, E_{1}f_{ij}+E_{2}g_{ij}\right\rangle \right\rbrace _{i\in\mathbb{I}, j\in\left[ m\right] }\\
&=&\left\lbrace \left\langle f, E_{1}f_{ij}\right\rangle +\left\langle f, E_{2}g_{ij}\right\rangle \right\rbrace _{i\in\mathbb{I}, j\in\left[ m\right]}\\
&=&\left\lbrace \left\langle E_{1}^{*}f, f_{ij}\right\rangle  \right\rbrace _{i\in\mathbb{I}, j\in\left[ m\right]}
+\left\lbrace \left\langle E_{2}^{*}f, g_{ij}\right\rangle \right\rbrace _{i\in\mathbb{I}, j\in\left[ m\right]}\\
&=& U_{F}E_{1}^{*}f+U_{G}E_{2}^{*}f,
\end{eqnarray*}
so the result is obtained.
\end{proof}
Now, by this theorem, we can immediately get the following corollary.
\begin{corollary}
By the hypotheses of Theorem \ref{sumT}, if $ U_{F}E_{1}^{*}+U_{G}E_{2}^{*} $ is an invertible operator on $ \mathcal{H} $, then the operator
$$ S=E_{1}S_{F}E_{1}^{*}+E_{2}S_{G}E_{2}^{*}+E_{1}U^{*}_{F}U_{G}E_{2}^{*}+ E_{2}U^{*}_{F}U_{F}E_{1}^{*} ,$$
is a positive operator.
\end{corollary}
\begin{proof}
By Theorem \ref{sumT}, the family of sequences $ \left\lbrace E_{1}f_{ij}+E_{2}g_{ij}\right\rbrace _{i\in\mathbb{I},j\in\left[ m\right] } $ is a woven frames with the analysis operator $ U_{F}E_{1}^{*}+U_{G}E_{2}^{*} $, for which its woven frame operator is positive:
\begin{eqnarray*}
Sf&=&\left( U_{F}E_{1}^{*}f+U_{G}E_{2}^{*}f\right) ^{*}\left( U_{F}E_{1}^{*}f+U_{G}E_{2}^{*}f\right)\\
&=& \left( E_{1}U^{*}_{F}f+E_{2}U^{*}_{G}f\right)
\left( U_{F}E_{1}^{*}f+U_{G}E_{2}^{*}f\right)\\
&=& E_{1}S_{F}E_{1}^{*}f+E_{1}U_{F}^{*}U_{G}E_{2}^{*}f+ E_{2}U_{G}^{*}U_{F}E_{1}^{*}f+E_{2}S_{G}E_{2}^{*}f.
\end{eqnarray*}
\end{proof}

In the following theorem, we check frame bounds conditions on a dense subset in Hilbert space:

\begin{theorem}
Suppose $ \left\lbrace f_{ij}\right\rbrace _{i\in\mathbb{I}, j\in\left[ m\right]} $ is a woven frames for $ W $ with universal woven bounds
$ C $ and $ D $, where $ W $ is a dense subspace of $ \mathcal{H} $. Then
$ \left\lbrace f_{ij}\right\rbrace _{i\in\mathbb{I}, j\in\left[ m\right]} $ is a woven frames for $ \mathcal{H} $, with universal bound
$ C $ and $ D $.
\end{theorem}
\begin{proof}
Since $ \left\lbrace f_{ij}\right\rbrace _{i\in\mathbb{I}, j\in\left[ m\right]} $ is woven frames for $ W $, then
\begin{eqnarray}
\label{T1}
C\Vert f\Vert^{2}\leq \sum_{i\in\mathbb{I}, j\in\left[ m\right] }\left\vert \left\langle f, f_{ij}\right\rangle \right\vert ^{2}\leq D\Vert f\Vert^{2},\quad \forall f\in W.
\end{eqnarray}
For upper bound, by contradiction, suppose that there exists $ g\in\mathcal{H} $, such that:
$$ \sum_{i\in\mathbb{I}, j\in\left[ m\right] }\left\vert \left\langle g, f_{ij}\right\rangle \right\vert ^{2}\geq D\Vert f\Vert^{2}. $$
By Corollary $ \ref{C1} $, the series
$ \sum_{i\in\mathbb{I}, j\in\left[ m\right] }\left\vert \left\langle g, f_{ij}\right\rangle \right\vert ^{2} $ is not convergent unconditionally. Then there exists a finite set $ M\subset\mathbb{I} $ where
$$ \sum_{i\in M, j\in\left[ m\right] }\left\vert \left\langle g, f_{ij}\right\rangle \right\vert ^{2}\geq D\Vert f\Vert^{2},\quad g\in\mathcal{H}.$$
Since $ W $ is a dens subset in $ \mathcal{H} $, so there exist $ h\in W $ such that
$$ \sum_{i\in M, j\in\left[ m\right] }\left\vert \left\langle h, f_{ij}\right\rangle \right\vert ^{2}\geq D\Vert f\Vert^{2},\quad h\in\mathcal{H}, $$
then the sequence $ \left\lbrace f_{ij}\right\rbrace _{i\in \sigma, j\in [m]} $ is not Bessel in $ W $ and this is a contradiction. So the family
$ \left\lbrace f_{ij}\right\rbrace _{i\in\mathbb{I}, j\in\left[ m\right]} $ is a Bessel woven in $ \mathcal{H} $. Now for lower frame bound, by $ (\ref{T1} ) $, we have for every $ f\in W $
\begin{eqnarray}
\label{T2}
C\Vert f\Vert^{2}\leq \sum_{i\in\mathbb{I}, j\in\left[ m\right] }\left\vert \left\langle f, f_{ij}\right\rangle \right\vert ^{2}
= \left\Vert U_{F}f\right\Vert^{2}.
\end{eqnarray}
Then $ U_{F} $ is a bounded operator in $ W $ for which is dens in $ \mathcal{H} $. Therefor the statement of $ (\ref{T2} ) $ holds for all $ f\in\mathcal{H} $.
\end{proof}
In the next theorem, from every woven frame in Hilbert space $ \mathcal{H} $, we constitute woven frame for smaller spaces, by using orthogonal projection.
\begin{theorem}
\label{T6}
Suppose $ \left\lbrace f_{ij}\right\rbrace _{i\in\mathbb{I}, j\in\left[ m\right]} $ is woven frame for Hilbert space $ \mathcal{H} $, with universal bounds $ C $ and $ D $ and $ P $ denotes the orthogonal projection onto a closed subspace $ W $ of $ \mathcal{H} $. Then $ \left\lbrace Pf_{ij}\right\rbrace _{i\in\mathbb{I}, j\in\left[ m\right]} $ is a woven frames for $ W $ with the same universal bounds.
\end{theorem}
\begin{proof}
Because $ \left\lbrace f_{ij}\right\rbrace _{i\in\mathbb{I}, j\in\left[ m\right]} $ is a woven frames for $ \mathcal{H} $, with universal bounds
$ C $ and $ D $, so we have:
\begin{eqnarray*}
C\Vert f\Vert^{2}\leq \sum_{i\in\mathbb{I}, j\in\left[ m\right] }\left\vert \left\langle f, f_{ij}\right\rangle \right\vert ^{2}\leq D\Vert f\Vert^{2},\quad \forall f\in\mathcal{H}.
\end{eqnarray*}
Since we have $ Pf=f $, for all $ f\in W $ and $ Pf=0 $, for all $ f\in W^{\perp} $, then for every $ f\in W $ we can write
\begin{eqnarray*}
\sum_{i\in\mathbb{I}, j\in\left[ m\right] }\left\vert \left\langle f, f_{ij}\right\rangle \right\vert ^{2}=
\sum_{i\in\mathbb{I}, j\in\left[ m\right] }\left\vert \left\langle Pf, f_{ij}\right\rangle \right\vert ^{2}=
\sum_{i\in\mathbb{I}, j\in\left[ m\right] }\left\vert \left\langle f, Pf_{ij}\right\rangle \right\vert ^{2}.
\end{eqnarray*}
Therefore $ \left\lbrace Pf_{ij}\right\rbrace _{i\in\mathbb{I}, j\in\left[ m\right]} $ is woven frames for $ W $ with bounds $ C $ and
$ D $.
\end{proof}
The following corollary follows from Theorem \ref{T6}.
\begin{corollary}
\label{C7}
Let $ \left\lbrace f_{ij}\right\rbrace _{i\in\mathbb{I}, j\in\left[ m\right]} $ be a woven frame for Hilbert space
$ \mathcal{H} $ and $ V, W $
are closed subspaces of $ \mathcal{H} $ such that $ V\cap W\neq \phi $ and $ P $ denotes the orthogonal projection of $ \mathcal{H} $ onto $ V\cap W $. Then $ \left\lbrace Pf_{ij}\right\rbrace _{i\in\mathbb{I}, j\in\left[ m\right]} $ is a woven frame for $ V\cap W $.
\end{corollary}
\begin{flushleft}
\newpage
{\LARGE Application:}
\end{flushleft}
In this section, we provide an example of woven frames in the Euclidean space $ \mathbb{R}^{3} $, then from this woven frame, we constitute a woven frame for smaller subspace of $ \mathbb{R}^{3} $, by Theorem $ \ref{T6} $ and Corollary $ \ref{C7}$.
\begin{example}
Let $ \left\lbrace e_{i}\right\rbrace _{i=1}^{3} $ be the standard orthonormal basis for $ \mathbb{R}^{3} $. Suppose there exist constants $ \alpha >0 $ and $ \beta=\frac{1}{\sqrt{1+\alpha^{2}}} $. Also $ G, Q $ are the sets
$$ G=\left\lbrace g_{i}\right\rbrace _{i=1}^{6}=
\left\lbrace \beta e_{1},\alpha\beta e_{1},\beta e_{2},\alpha\beta e_{2},\beta e_{3},\alpha\beta e_{3},\right\rbrace  $$
and:
$$ Q=\left\lbrace q_{i}\right\rbrace _{i=1}^{6}=\left\lbrace \alpha\beta e_{1},\beta e_{1},\alpha\beta e_{2},\beta e_{2},\alpha\beta e_{3},\beta e_{3},\right\rbrace . $$
 $ G $ is a Parseval frame:
 \begin{eqnarray*}
 \sum_{i=1}^{6}\vert\left\langle f, g_{i}\right\rangle \vert^{2}&=&\vert\left\langle f, \beta e_{1}\right\rangle \vert^{2}+\vert\left\langle f, \alpha\beta e_{1}\right\rangle \vert^{2}\\
&+&\vert\left\langle f, \beta e_{2}\right\rangle \vert^{2}+\vert\left\langle f, \alpha\beta e_{2}\right\rangle \vert^{2}+\vert\left\langle f, \beta e_{3}\right\rangle \vert^{2}+\vert\left\langle f, \alpha\beta e_{3}\right\rangle \vert^{2}\\&=& \beta^{2}\left( 1+\alpha^{2}\right) \sum_{i=1}^{3}\vert\left\langle f, e_{i}\right\rangle \vert^{2}\\
 &=& \Vert f\Vert^{2}.
\end{eqnarray*}
Similarly $ Q $ is a Parseval frame, for which $ G, Q $ are woven frames, since every
$ \sigma\subset \left\lbrace 1,2,3,4,5,6\right\rbrace  $
gives a spanning set. For example, we calculate for two weaving:
Let $ \sigma_{1} =\left\lbrace 2,4,6\right\rbrace  $. Then for every $ f\in\mathbb{R}^{3} $, we have:
 \begin{eqnarray*}
 \sum_{i\in\sigma_{1}}\vert\left\langle f, g_{i}\right\rangle \vert^{2}+\sum_{i\in\sigma_{1}^{c}}\vert\left\langle f,q_{i}\right\rangle \vert^{2}&=& \alpha^{2}\beta^{2}\sum_{i=1}^{3}\vert\left\langle f, e_{i}\right\rangle \vert^{2}\\&=&\alpha^{2}\beta^{2}\Vert f\Vert^{2}\\&=& \frac{\alpha^{2}}{1+\alpha^{2}}\Vert f\Vert^{2},
 \end{eqnarray*}
 So this weaving is tight frame with bounds
 $ C_{\sigma_{1}}=D_{\sigma_{1}}=\frac{\alpha^{2}}{1+\alpha^{2}} $.
Also for $ \sigma_{2}=\left\lbrace 1,2 \right\rbrace  $, we have:
\begin{eqnarray*}
 \sum_{i\in\sigma_{2}}\vert\left\langle f, g_{i}\right\rangle \vert^{2}+\sum_{i\in\sigma_{2}^{c}}\vert\left\langle f, q_{i}\right\rangle \vert^{2} &=&\vert\left\langle f, \beta e_{1}\right\rangle \vert^{2}+\vert\left\langle f, \alpha\beta e_{1}\right\rangle \vert^{2}+\vert\left\langle f, \alpha\beta e_{2}\right\rangle \vert^{2}\\
 &+&\vert\left\langle f,\beta e_{2}\right\rangle \vert^{2}+\vert\left\langle f, \alpha\beta e_{3}\right\rangle \vert^{2}+
 \vert\left\langle f, \beta e_{3}\right\rangle \vert^{2}\\
 &=& \left( \alpha^{2}\beta^{2}+\beta^{2}\right)\sum_{i=1}^{3}\vert\left\langle f, e_{i}\right\rangle \vert^{2}\\
&=&\Vert f\Vert^{2},
\end{eqnarray*}
thus this weaving is Parseval frame and $ C_{\sigma_{2}}=D_{\sigma_{2}}=1 $ . Now, If we have:
 $$ C=\min\left\lbrace C_{\sigma_{j}}\quad s.t\quad j\in\left[64\right] \right\rbrace $$
 and
 $$D=\max\left\lbrace D_{\sigma_{j}}\quad s.t\quad j\in\left[64\right] \right\rbrace , $$
therefor $ G,Q $ are woven frames with universal bounds $ C $ and $ D $.
Now, let $ V_{1}=\overline{\rm span}\left\lbrace e_{3i}, e_{3i+1}\right\rbrace  $ and $ P $ denotes the orthogonal projection from
$ \mathbb{R}^{3} $ onto $ V_{1} $. Then by Theorem $ \ref{T6} $,
$ \left\lbrace Pg_{i}\right\rbrace_{i=1}^{6}  $ and $ \left\lbrace Pq_{i}\right\rbrace_{i=1}^{6}  $ are woven frames for $ V_{1} $.\\
Also, suppose $ V_{2}=\overline{\rm span}\left\lbrace e_{3i+1}, e_{3i+2}\right\rbrace $. Then
$ V_{1}\cap V_{2}=\overline{\rm span}\left\lbrace e_{3i+1}\right\rbrace $. Let
$ P^{\prime} $ be an orthonormal projection of $ \mathbb{R}^{3} $ onto $ \overline{\rm span}\left\lbrace e_{3i+1}\right\rbrace $.
Thus by Corollary $ \ref{C7}$,
$ \left\lbrace P^{\prime}g_{i}\right\rbrace_{i=1}^{6}  $ and $ \left\lbrace P^{\prime}q_{i}\right\rbrace_{i=1}^{6}  $ are woven frames for
$ V_{1}\cap V_{2} $ with same bounds $ C $ and $ D $.
\end{example}


\bibliographystyle{amsplain}

\end{document}